\documentclass[10pt]{article}
\usepackage{makeidx}
\usepackage{multirow,a4wide}
\usepackage{multicol}
\usepackage[dvipsnames,svgnames,table]{xcolor}
\usepackage{graphicx}
\usepackage{epstopdf}
\usepackage{hyperref}
\usepackage{amsmath}
\usepackage{mathrsfs}
\usepackage{amssymb}
\usepackage{hyperref}
\usepackage{makeidx}\usepackage{esvect}
\usepackage{multirow}
\usepackage{multicol}
\usepackage[dvipsnames,svgnames,table]{xcolor}
\usepackage{graphicx}
\usepackage{epstopdf}
\usepackage{hyperref}
\usepackage{amsmath,tipa}
\usepackage{amssymb}
\usepackage{amsthm,bm}
\author{Matthieu Calvez}
\title{Euclidean Artin-Tits groups are acylindrically hyperbolic}

\newtheorem{theorem}{Theorem}[section]
\newtheorem{lemma}[theorem]{Lemma}
\newtheorem{proposition}[theorem]{Proposition}
\newtheorem*{TheoremA}{Theorem A}
\newtheorem*{CorollaryB}{Corollary B}

\theoremstyle{definition}

\newtheorem{definition}[theorem]{Definition}

\newtheorem{example}[theorem]{Example}
\newtheorem{remark}[theorem]{Remark}

\newtheorem{claim}[theorem]{Claim}
\newtheorem{Proposition-Definition}{Proposition-Definition}[section]

\def\Dir{\text{D{\scriptsize\scshape{IR}}}}
\def\Dim{\text{D{\scriptsize\scshape{IM}}}}
\def\Mov{\text{M{\scriptsize\scshape{OV}}}}
\def\Min{\text{M{\scriptsize\scshape{IN}}}}

\def\Isom{\text{I{\scriptsize\scshape{SOM}}}}
\def\dis{\text{\scshape{dis}}}

\begin{document}

\maketitle

\begin{abstract}
In this paper we show the statement in the title. To any Garside group of finite type, Wiest and the author associated a hyperbolic graph called the \emph{additional length graph} and they used it to show that central quotients  of Artin-Tits groups of spherical type are acylindrically hyperbolic. In general, a euclidean Artin-Tits group is not \emph{a priori} a Garside group but McCammond and Sulway have shown that it embeds into an \emph{infinite-type} Garside group which they call a {\em{crystallographic Garside group}}. We associate a \emph{hyperbolic} additional length graph to this crystallographic Garside group and we exhibit elements of the euclidean Artin-Tits group which act loxodromically and WPD on this hyperbolic graph.
\end{abstract}

\section{Introduction}
An \emph{Artin-Tits group} is a group defined by a presentation involving a finite set of generators $S$ (the \emph{standard generators}) and where all the relations are as follows: every pair $(a,b)$ of standard generators satisfies at most one balanced relation of the form: 
$$\Pi(a,b;m_{a,b})=\Pi(b,a;m_{a,b}),$$
with $m_{a,b}=m_{b,a}\geqslant 2$ and where for $j\geqslant 2$, 
$$\Pi(a,b;j)=\begin{cases} (ab)^{\frac{j}{2}} & \text{if $j$ is even,}\\
(ab)^{\frac{j-1}{2}} a & \text{if $j$ is odd.}\\
\end{cases}$$

We also write $m_{a,b}=m_{b,a}=\infty$ when $a$ and $b$ satisfy no relation. This presentation can be encoded by a \emph{Coxeter graph} $\Gamma$. The vertices of $\Gamma$ are in bijection with the set $S$. Two distinct vertices $a,b$ of $\Gamma$ are connected by an edge labeled $m_{a,b}$ if and only if either they satisfy no relation or $m_{a,b}>2$. The Artin-Tits group defined by the Coxeter graph $\Gamma$ will be denoted by $A_{\Gamma}$. The \emph{rank} of $A_{\Gamma}$ is the cardinality of $S$. 
The quotient of $A_{\Gamma}$ by the normal subgroup generated by the squares of the elements in $S$ is a \emph{Coxeter group} denoted by $W_{\Gamma}$. The group $A_{\Gamma}$ ($W_{\Gamma}$ respectively) is said to be \emph{irreducible} if $\Gamma$ is connected. 

The geometry of Artin-Tits groups has recently attracted a lot of attention and non-positive curvature features have been exhibited for many classes of Artin-Tits groups. In this paper, we focus on the acylindrical hyperbolicity, which was introduced and extensively discussed in \cite{Osin}. The class of acylindrically hyperbolic groups is both sufficiently large to include a significative number of interesting groups and restrictive enough to deduce many interesting consequences.
An isometric action of a group $G$ on a metric space $(X,d_X)$ is \emph{acylindrical} if 
for every $\epsilon>0$, there exist $R,N>0$ such that whenever two points $x,y\in X$ are at a distance at least $R$ apart, then
$$\#\{g\in G,\  d_X(x,g\cdot x)\leqslant \epsilon, d_{X}(y,g\cdot y)\leqslant \epsilon\}\leqslant N.$$
A group is \emph{acylindrically hyperbolic} if it is not virtually cyclic and admits an acylindrical isometric action on a hyperbolic metric space with unbounded orbits.

It is conjectured that the central quotient of every irreducible Artin-Tits group of rank at least~2 is acylindrically hyperbolic \cite[Conjecture B]{Haettel}. Artin-Tits groups of rank 2 are called \emph{dihedral} Artin-Tits groups and (when irreducible) their central quotients are virtually free hence acylindrically hyperbolic \cite{Haettel}. 
Here is a brief overview of some classes of Artin-Tits groups for which the conjecture has been proved. 

\begin{itemize}
\item Artin's braid group, seen as the Mapping Class Group of the punctured disk \cite{Bowditch}.
\item Artin-Tits groups of spherical type (the Coxeter group is finite) \cite{CalvezWiestAH}. 
\item Right-Angled Artin-Tits groups ($m_{a,b}\in\{2,\infty\}$ for any standard generators $a\neq b$)~\cite{KimKoberda}.
\item Artin-Tits groups of extra extra large type (${m_{a,b}\geqslant 5}$ for any standard generators $a\neq b$)~\cite{Haettel}.
\item 2-dimensional Artin-Tits groups of hyperbolic type (for each triple of distinct standard generators $a,b,c$, we have $\frac{1}{m_{a,b}}+\frac{1}{m_{b,c}}+\frac{1}{m_{a,c}}\leqslant1$, and the associated Coxeter group is hyperbolic)~\cite{MartinPrz}.
\item Artin-Tits groups for which there is no partition $S=S_1\sqcup S_2$ of the set of standard generators satisfying
$m_{a,b}<\infty$ for all $a\in S_1, b\in S_2$
  \cite{CharneyMorris}.
\item 2-dimensional Artin-Tits groups \cite{Vaskou}. 
\end{itemize}

In this paper, we focus on the class of \emph{euclidean Artin-Tits groups}. The theory of arbitrary Coxeter groups and Artin-Tits groups stems from the study of discrete groups generated by reflections which act geometrically on spheres (finite Coxeter groups) and euclidean spaces (euclidean Coxeter groups). Finite and euclidean Coxeter groups are central in Lie theory and they were studied much longer before Tits gave the general definition  \cite{Tits} of arbitrary Coxeter and Artin-Tits groups. An Artin-Tits group has \emph{spherical type}, or \emph{euclidean type}, respectively, if the associated Coxeter group is finite, or euclidean, respectively. 

There is a well-known classification of connected Coxeter graphs ((extended) Dynkin diagrams) defining irreducible finite and euclidean Coxeter groups. Artin-Tits groups of spherical type have long been well-understood thanks to their Garside structure \cite{Deligne,BrieskornSaito}. By contrast, Artin-Tits groups of euclidean type remained mostly mysterious (with some exceptions --see \cite{Squier,DigneA,DigneC}) for decades, until their structure was elucidated by McCammond and Sulway~\cite{McCS} around 2015. In this paper we prove the above conjecture for irreducible Artin-Tits groups of euclidean type (note that these groups are centerless, by \cite[Proposition 11.9]{McCS}):
\begin{TheoremA} 
Let $A$ be an irreducible Artin-Tits group of euclidean type. Then $A$ is acylindrically hyperbolic.
\end{TheoremA}

To establish that a given group is acylindrically hyperbolic, it is often simpler to use an equivalent characterization due to Osin \cite[Theorem 1.2]{Osin}: a non-virtually cyclic group $G$ is acylindrically hyperbolic if and only if it acts by isometries on a hyperbolic metric space and \emph{some} element acts in a loxodromic Weakly Properly Discontinuous (WPD) fashion --see \cite{BestvinaFujiwara}. An element $g\in G$ acts \emph{loxodromically} on the metric space $(X,d_X)$ if for some (any) $x\in X$ there is some $\kappa>0$ so that for all $k\in \mathbb Z$,
$$d_X( x,g^k\cdot x)\geqslant |k|\kappa$$ and \emph{Weakly Properly Discontinuously} (WPD) if for all $x\in X$ and for all $\epsilon>0$, there exists $N>0$ such that the set 
$$\{h\in G,\ d_X(x,h\cdot x)\leqslant\epsilon, d_X(g^N\cdot x, hg^N\cdot x)\leqslant \epsilon\}$$ 
is finite. 

Consider first the case of the \emph{affine braid group}, or Artin-Tits group of type $\widetilde A_n$. It is known that it embeds into the central quotient of Artin's braid group \cite{KentPeifer}, which has an acylindrical action on the curve graph of the punctured disk \cite{Bowditch}. This curve graph is hyperbolic and each pseudo-Anosov braid acts on it in a loxodromic way \cite{MasurMinsky}. Because the action is acylindrical, it follows that each pseudo-Anosov braid acts on the curve graph in a WPD manner. It is not difficult to find such a pseudo-Anosov braid in the image of the affine braid group; therefore, as the affine braid group is not virtually cyclic, the above mentioned result by Osin applies to show that the affine braid group is acylindrically hyperbolic. 
 
Our proof for a general irreducible Artin-Tits group of euclidean type closely follows the construction of Wiest and the author to show that central quotients of Artin-Tits groups of spherical type are acylindrically hyperbolic \cite{CalvezWiestAH}. Let us recall the strategy. They first constructed from any finite-type Garside group $G$ a \emph{hyperbolic} graph $\mathcal C_{AL}(G)$ called the \emph{additional length graph} on which the group $G/ZG$ acts isometrically, see \cite{CalvezWiestCurve}. When $G$ is an Artin-Tits group of spherical type, they exhibit in~\cite{CalvezWiestAH} some element $x_G$ of $G/ZG$ whose action on this hyperbolic graph is loxodromic and WPD. Acylindrical hyperbolicity of $G/ZG$ then follows by the above-mentioned theorem of Osin. 

On another hand, McCammond and Sulway have established that each irreducible Artin-Tits group of euclidean type $A$ embeds in a so-called \emph{crystallographic group} $\mathfrak C$ with a Garside structure of \emph{infinite-type} \cite{McCS}; such a group is sometimes also called a \emph{quasi-Garside group}. It turns out that the construction and the proof of the hyperbolicity of the additional length graph given in \cite{CalvezWiestCurve} adapt immediately in this more general context and we obtain again a hyperbolic graph $\mathcal C_{AL}(\mathfrak C)$ with an isometric action of $\mathfrak C$. 
 Then, we construct elements of $A<\mathfrak C$ which act loxodromically and in a WPD fashion on this additional length graph. 
We conclude in the same way, using Osin's characterization of acylindrically hyperbolic groups.
 
 The paper is organized as follows. In Section \ref{S:Garside}, we give the suitable definition of a Garside group, we recall the construction of the additional length graph and its hyperbolicity. In Section \ref{S:Euclidean}, we recall a number of facts on euclidean Coxeter groups used in the sequel. In Section~\ref{S:Loxodromic}, we construct the desired loxodromic elements.
 In Section \ref{S:Proofs}, we prove Theorem A.

\section{Garside structure and the additional length graph}\label{S:Garside}

The reader is referred to \cite{DDGKM} for a detailed account on Garside theory; the unpublished text \cite{McCGarside} can also be useful. 

\begin{definition}[Garside monoid]\label{D:Garside}

A monoid $M$ is a \emph{Garside monoid} if it satisfies the following conditions: 

\begin{itemize}
\item[(1)] $M$ is left and right cancellative, that is, for all $a,b,c\in M$, either of the conditions $ab=ac$ or $ba=ca$ implies $b=c$. 
\item[(2)] There exists a map $\rho: M \longrightarrow \mathbb N\cup \{0\}$ satisfying $\rho(ab)\geqslant\rho(a)+\rho(b)$ for all $a,b\in M$ and $\rho(a)=0$ if and only if $a=1$. 

\item[(3)] Both relations in $M$
\begin{itemize}
\item $a\preccurlyeq b$ if and only if there is $c\in M$ so that $b=ac$ ($a$ is a \emph{prefix} or \emph{left divisor} of $b$ or $b$ is a \emph{right multiple} of $a$) 
\item $a\succcurlyeq b$ if and only if there is $c\in M$ so that $a=cb$ ($b$ is a \emph{suffix} or \emph{right divisor} of $a$ or $a$ is a \emph{left multiple} of $b$)
\end{itemize}
are lattice orders on $M$. 
\item[(4)] There is an element $\Delta\in M$, called  the \emph{Garside element}, such that the left and right divisors of $\Delta$ are the same and generate  $M$. These elements are called \emph{simple elements}. A simple element is \emph{proper} if it is distinct from 1 and $\Delta$.
\end{itemize}
\end{definition}

\begin{remark}\label{R:Garside}
(i) In the usual definition of a Garside monoid, the set of simple elements is assumed to be finite; a Garside monoid with a finite number of simple elements is said to be of \emph{finite type}. In absence of this condition, the monoid $M$ is also sometimes called a \emph{quasi}-Garside monoid --see \cite[Definitions 2.1 and 2.2]{DDGKM}. 

(ii) For some of the Garside monoids considered in this paper, it will be convenient to modify slightly the condition (2), 
so that the function $\rho$ has values in $\mathbb N\cup \frac{2}{3}\mathbb N\cup\{0\}$ instead of $\mathbb N\cup\{0\}$. 

(iii) An element $a\in M$ is an \emph{atom} if $a$ is indecomposable, that is if the condition $a=bc$, with $b,c\in M$ implies $b=1$ or $c=1$. The set of atoms generates $M$.
\end{remark}

\begin{definition}[gcds, lcms, complements and weightedness]\label{D:Complements}
The \emph{left/right greatest common divisor} of $a,b\in M$ is denoted by $a\wedge b / a\wedge ^{\Lsh} b$; the right/left least common multiple of $a,b$ is denoted by $a\vee b / a\vee ^{\Lsh} b$. Let $s$ be a simple element. Owing to Condition (1), there exists a \emph{unique} $\partial(s)\in M$ such that $s\partial(s)=\Delta$. By Condition (4), this element 
is still a simple element, called the \emph{right complement} of~$s$. Similarly, we have the \emph{left complement} $\partial^{-1}(s)$ of $s$, which is the unique simple element satisfying $\partial ^{-1}(s)s=\Delta$. Conjugation by~$\Delta$ induces a bijection of the set of simple elements which we denote by $\tau$: $\tau(s)=\partial(\partial(s))$ and $s\Delta=\Delta\tau(s)$. 
This map extends to an automorphism of $M$; when $M$ is of finite type, this automorphism has finite order but this need not be the case in our more general context. 
An ordered pair $(s,s')$ of simple elements is \emph{left-weighted} if $\partial s\wedge s'=1$. In other words, this means that no non-trivial prefix $a$ of $s'$ satisfies that $sa$ is still a simple element. 
Similarly, the ordered pair $(s,s')$ of simple elements is called \emph{right-weighted} if $\partial^{-1}(s')\wedge^{\Lsh} s=1$. 
\end{definition}

\begin{proposition}[Normal forms]
A Garside monoid $M$ embeds in its group of fractions $G$ and~$G$ is called a \emph{Garside group}; in this context, the elements of $M$ are called \emph{positive}. Each element $g$ in $G$ admits a unique decomposition $g=a^{-1}b$, where $a,b\in M$ and $a\wedge b=1$. This is called the \emph{negative-positive normal form} of $g$. 
Each element $g\in G$ also admits a unique decomposition $g=\Delta^p s_1\ldots s_q$, called its \emph{left normal form}, where $p\in\mathbb Z$, $q\geqslant 0$, and
the $s_i$ are \emph{proper} simple elements such that $(s_i,s_{i+1})$ is left-weighted.
Similarly, each $g\in G$ admits a unique \emph{right normal form} $g=s'_q\ldots s'_1 \Delta^p$ where the $s'_i$ are \emph{proper} simple elements such that $(s'_{i+1},s'_{i})$ is right-weighted.
The integers $p$, $q$ and $q+p$ are respectively called the \emph{infimum}, the \emph{canonical length} and the \emph{supremum} of $g$ and we denote $p=\inf(g)$, $q=\ell(g)$ and $p+q=\sup(g)$. 
\end{proposition}

The following statements can all be found in \cite{CalvezWiestCurve} for \emph{finite-type} Garside groups. However, they extend immediately to our more general context.

\begin{definition}[Absorbable] See \cite[Definition 1]{CalvezWiestCurve}.
Let $G$ be a Garside group. 
An element $g$ of~$G$ is said to be \emph{absorbable} if the two following conditions are satisfied: 
\begin{itemize}
\item $\inf(g)=0$ or $\sup(g)=0$,
\item there exists some $h\in G$ such that $$\begin{cases} \inf(hg)=\inf(h) \ \ \text{and}\\ \sup(hg)=\sup(h).\end{cases}$$
\end{itemize}
\end{definition}

The following lemma is a useful technical fact about absorbable elements. 
\begin{lemma}[Subwords of absorbable elements]\label{L:AbsorbableSubword} {\rm See \cite[Lemma 2]{CalvezWiestCurve}.}
Let $G$ be a Garside group. Suppose that an absorbable element~$g\in G$ factors as a product of positive (possibly trivial) elements: $g=g_1g_2g_3$. Then $g_2$ is absorbable. 
\end{lemma}

\begin{definition}[Additional length graph] See \cite[Definition 2]{CalvezWiestCurve}, see also \cite[Section 2.1]{Bestvina}.
Let $G$ be a Garside group. The \emph{additional length graph} associated to $G$ is the graph denoted by $\mathcal C_{AL}(G)$ defined in the following way: 

\begin{itemize}
\item The vertices are in bijection with the left cosets $g\Delta^{\mathbb Z}, g\in G$. Each vertex $V$ has a unique \emph{distinguished representative} $\underline{V}$ of infimum 0. 
We denote by $\ast$ the vertex $\Delta^{\mathbb Z}$.
\item Two vertices $V$ and $V'$ are connected by an edge if and only if one of the following holds: 
\begin{itemize}
\item there is a proper simple element $s$ such that $\underline{V}s$ belongs to the coset $V'$ (this is equivalent to saying that there is some proper simple element $s'$ so that $\underline{V'} s'$ belongs to the coset~$V$). 
\item there is an absorbable element $g$ such that $\underline{V}g\in V'$ (equivalently, there is an absorbable element $g'$ so that $\underline{V'}g'\in V$). 
\end{itemize}
\end{itemize}
The graph is endowed with the edge-metric which we denote by $d_{AL}$. There is an isometric action of $G$ by left translation on the vertices: for $g\in G$ and $V$ a vertex of $\mathcal C_{AL}(G)$, $g\cdot V=(g\underline{V})\Delta^{\mathbb Z}$.  
\end{definition}

\begin{definition}[Preferred paths] See \cite[Definition 3]{CalvezWiestCurve}, see also \cite[Section 3.1]{Bestvina}. Let $G$ be a Garside group. Consider a vertex $V$ of $\mathcal C_{AL}(G)$ and write $\underline{V}=s_1\ldots s_r$ the left normal form of its distinguished representative. 
The \emph{preferred path} $\mathcal A(\ast,V)$ is the path 
$$\ast,\ s_1\Delta^{\mathbb Z},\ \ldots \ ,\ (s_1\ldots s_r)\Delta^{\mathbb Z}=V$$
from $\ast$ to $V$.
Given any two vertices $V_1, V_2$ of $\mathcal C_{AL}(G)$, the preferred path $\mathcal A(V_1,V_2)$ is the $\underline {V_1}$ left translate of the path 
$\mathcal A(\ast, \underline {V_1}^{-1}\cdot V_2)$. 
\end{definition} 

Here is a summary of the properties enjoyed by the preferred paths. 

\begin{proposition}[Properties of preferred paths]\label{P:Preferred}
Let $G$ be a Garside group.
\begin{itemize}
\item[(i)] {\rm{\cite[Lemma 4]{CalvezWiestCurve}.}}  Let $V_1$ and $V_2$ be two vertices of $\mathcal C_{AL}(G)$. The preferred path $\mathcal A(V_1,V_2)$ is the concatenation of the paths 
$\mathcal A(V_1, (\underline{V_1}\wedge \underline{V_2}) \Delta^{\mathbb Z})$ and $\mathcal A((\underline{V_1}\wedge\underline{V_2})\Delta^{\mathbb Z},V_2)$. 
\item[(ii)] {\rm{\cite[Lemma 5]{CalvezWiestCurve}.}} The preferred paths are symmetric; that is,  for all vertices $V_1$ and $V_2$ of $\mathcal C_{AL}(G)$, $\mathcal A(V_2,V_1)$ is the reverse of $\mathcal A(V_1,V_2)$. 
\item[(iii)] {\rm{\cite[Lemma 2]{CalvezWiestAH}.}} Let $V_1,V_2$ be two vertices of $\mathcal C_{AL}(G)$; let $g\in G$. We have $\mathcal A(g\cdot V_1, g\cdot V_2)=g\cdot \mathcal A(V_1,V_2)$. 
\item[(iv)] {\rm{\cite[Lemma 7]{CalvezWiestCurve}.}} Let $V_1,V_2,V_3$ be three vertices of $\mathcal C_{AL}(G)$. The triangle in $\mathcal C_{AL}(G)$ with vertices $V_1, V_2, V_3$ and with sides $\mathcal A(V_1,V_2)$, $\mathcal A(V_2,V_3)$ and $\mathcal A(V_3, V_1)$ is 2-thin: each side is at Hausdorff distance at most 2 from the union of the other two sides.
\end{itemize}
\end{proposition}

Finally, the main result of \cite{CalvezWiestCurve} is the following:

\begin{theorem}[{\rm{Hyperbolic}}] {\rm{See \cite[Theorem 1]{CalvezWiestCurve}.}}\label{T:CAL} Let $G$ be a Garside group.
The graph $\mathcal C_{AL}(G)$ is 60-hyperbolic and the preferred paths form a family of 
uniformly unparameterized quasi-geodesics: for all vertices $V_1,V_2$ of $\mathcal C_{AL}(G)$, the Hausdorff distance between $\mathcal A(V_1,V_2)$ and any geodesic connecting $V_1$ and $V_2$ is bounded above by 39. 
\end{theorem}


\section{Euclidean Coxeter groups}\label{S:Euclidean}

In this section we gather a number of useful facts concerning euclidean Coxeter groups. We follow McCammond's approach developped in \cite{BMcC, McCFailure, McCS} with Brady and Sulway, see also \cite{McCSurvey}. 
Other useful references are \cite{Bourbaki,Humphreys}.

\subsection{Euclidean isometries}
\begin{definition}[Euclidean space and its isometries]
We denote by $\mathbb E=\mathbb R^n$ the $n$-dimensional euclidean space endowed with the usual scalar product 
$$\langle \eta,\eta'\rangle=\langle (\eta_i)_{i=1}^n , (\eta'_i)_{i=1}^n \rangle=\sum_{i=1}^n \eta_i\eta'_i.$$
 Two elements 
$\eta,\eta'\in \mathbb E$ are \emph{orthogonal} if $\langle \eta,\eta'\rangle = 0$. 
Given $\eta,\eta'\in \mathbb E$, the \emph{distance}  between $\eta$ and~$\eta'$ is the euclidean norm 
$||\eta'-\eta||=\sqrt{\langle \eta'-\eta,\eta'-\eta\rangle}$.
We denote by $\Isom(\mathbb E)$ the group of isometries of~$\mathbb E$, that is, the group of distance-preserving transformations of $\mathbb E$. 
Throughout, the elements of~$\mathbb E$ will be formally considered as vectors; however, 
$\mathbb E$ can be identified with the affine space which it underlies and sometimes it will be intuitively clearer to think of elements of~$\mathbb E$ also as \emph{points} rather than vectors. 
\end{definition}

\begin{definition}[Linear subspace] A \emph{linear subspace} of $\mathbb E$ is a non-empty subset closed under linear combination. 
We denote by $\Dim(U)$ the dimension of a linear subspace $U$ of $\mathbb E$. 
Each linear subspace~$U$ of $\mathbb E$ has an \emph{orthogonal complement} $U^{\perp}$, which is the linear subspace of $\mathbb E$ made of those elements in $\mathbb E$ which are orthogonal to all elements of $U$. There is a direct sum decomposition $\mathbb E=U\oplus U^{\perp}$ and $\Dim(U^{\perp})=n-\Dim(U)$ is called the \emph{codimension} of $U$. 
\end{definition}

\begin{definition}[Affine subspace]
A subset $B$ of $\mathbb E$ is an \emph{affine subspace} if there is a linear subspace~$U$ of $\mathbb E$ and an element $\theta$ of $\mathbb E$ such that $B=U+\theta$. Note that $U=\{\eta'-\eta, \ \eta,\eta'\in B\}$. 
The linear subspace $U$ is called the \emph{direction} of the affine subspace $B$ and we denote it by 
$\Dir(B)$. 
 The \emph{dimension} of~$B$ is $\Dim(\Dir(B))$. An affine subspace is a linear subspace if and only if it contains $0_{\mathbb E}$ (equivalently if it is equal to its direction). Given an affine subspace $B$, there is a unique $\theta_0\in B$ such that $||\theta_0||$ is minimal; then $B=\Dir(B)+\theta_0$ and this is called the \emph{standard form} of $B$ --note that $\theta_0\in \Dir(B)^{\perp}$. Two affine subspaces $B_1,B_2$ of $\mathbb E$ are \emph{parallel} if $B_1\cap B_2=\emptyset$ and $\Dir(B_1)\subset \Dir(B_2)$ (or vice-versa). 
 \end{definition}

%
%
%

\begin{definition}[Hyperplane and reflection]
A \emph{hyperplane} of $\mathbb E$ is an affine subspace of dimension $n-1$. Given a hyperplane $H$ in $\mathbb E$, there is a unique non-trivial isometry $r_H$ of $\mathbb E$ which fixes~$H$ pointwise; $r_H$ is called the \emph{reflection through~$H$}. The orthogonal complement of $\Dir(H)$ is a line in $\mathbb E$; a non-trivial vector in this line is called a \emph{root} of the reflection $r_H$. 
Given a hyperplane $H\subset \mathbb E$ and a root $\alpha$ of $r_H$, there is a unique $c\in \mathbb R$ such that $H=\{\eta\in \mathbb E, \langle \eta, \alpha\rangle=c\}$. With the same notation, we also write $r_{\alpha,c}=r_H$. This isometry can be described explicitly by 
$$r_{\alpha,c}(\eta)=\eta-2\frac{\langle\alpha,\eta\rangle-c}{\langle\alpha,\alpha\rangle}\alpha,\ \ \forall \eta\in \mathbb E.$$
\end{definition}

\begin{definition}[Translation]
Given $\lambda\in \mathbb E$, the map $t_{\lambda}: \eta\mapsto \eta+\lambda$ is called the \emph{translation of vector $\lambda$}; this transformation belongs to $\Isom(\mathbb E)$.  
The subset of all translations in $\Isom(\mathbb E)$ is an abelian subgroup isomorphic to the additive group of $\mathbb E$. 
\end{definition}

\begin{definition}[Basic invariants]
Associated to any isometry $u\in \Isom(\mathbb E)$ are two basic invariants, called the \emph{move-set} and the \emph{min-set} of $u$ \cite[Definition 3.1]{BMcC}. 
Given $\eta\in\mathbb E$, 
its \emph{displacement} under $u$ is $\dis_u(\eta)=u(\eta)-\eta$ and the move-set of $u$ is the set of displacements: $\Mov(u)=\{\dis_u(\eta), \eta\in \mathbb E\}$.
For every $u\in \Isom(\mathbb E)$, $\Mov(u)$ is an affine subspace of $\mathbb E$; the min-set of $u$ is the set $\Min(u)=\{\eta\in \mathbb E, ||\dis_u(\eta)||\  \text{minimal}\}$ and this is also an affine subspace of $\mathbb E$ 
\cite[Proposition 3.2]{BMcC}. Note that $\Mov(u)$ is a linear subspace if and only if $u$ has some fixed point, in which case $u$ is called \emph{elliptic}. Otherwise, $u$ is called \emph{hyperbolic}. For each isometry $u$, the respective directions of the move-set and the min-set of $u$ are mutually orthogonal complementary linear subspaces of $\mathbb E$ \cite[Lemma 3.6]{BMcC}. Intuitively, the move-set of $u$ is the set of vectors which are motions of points under the isometry $u$ while the min-set of $u$ is the set of points with the minimal possible motion. 
\end{definition}


\begin{example}
(i) Let $H$ be a hyperplane in $\mathbb E$. The reflection $r_H$ is an elliptic isometry, with $\Min(r_H)=H$ and $\Mov(r_H)= (\Dir(H))^{\perp}$. 
(ii) If $\lambda\in \mathbb E$ is non-zero, the translation of vector~$\lambda$ is a hyperbolic isometry with $\Min(t_{\lambda})=\mathbb E$ and $\Mov(t_{\lambda})=\{\lambda\}$. 
(iii) Let $H\subset \mathbb E$ be a hyperplane; a \emph{glide-reflection} through $H$ is the composition of the reflection $r_H$ and a translation $t_{\lambda}$ by a non-zero vector $\lambda\in \Dir(H)$; this is a hyperbolic isometry whose min-set is $H$ and whose move-set is the affine line $(\Dir(H))^{\perp}+\lambda$. 
\end{example}

%


\begin{definition}[Reflection length and order] By Cartan-Dieudonn\'e Theorem, the reflections generate $\Isom(\mathbb E)$. Given $u\in \Isom(\mathbb E)$, its \emph{reflection length} $|u|_{\Isom}$ is the minimal number of reflections needed to write $u$. Note that conjugates of reflections are again reflections, so that the reflection length is invariant under conjugacy. We have a partial order on $\Isom(\mathbb E)$ given by $u\preccurlyeq u'$ if and only if $|u|_{\Isom}+|u^{-1}u'|_{\Isom}=|u'|_{\Isom}$. Given $v\in \Isom(\mathbb E)$, we denote by $[1,v]=[1,v]^{\Isom(\mathbb E)}$ the \emph{interval} formed by those isometries $u$ satisfying $u\preccurlyeq v$. 
\end{definition}

The main result in \cite{BMcC} gives a close relation between this partial order and the basic invariants. Here, we record only what will be used in the sequel: 

\begin{proposition}\label{P:Poset}
\begin{itemize}
\item[(i)]  A hyperbolic isometry can never be smaller than an elliptic one. 
\item[(ii)] {\rm{\cite[Theorem 8.7]{BMcC}.}}
Let $v\in \Isom(\mathbb E)$. 
If $u_1, u_2\in [1,v]$ are elliptic isometries, then $u_1\preccurlyeq u_2$ if and only if $\Min(u_2)\subset \Min(u_1)$.
\end{itemize}
\end{proposition}
%
%


\subsection{Euclidean Coxeter groups}
\begin{definition}[Euclidean Coxeter groups]
Irreducible euclidean Coxeter groups (hence also irreducible Artin-Tits groups of euclidean type) are 
classified into four infinite families and five exceptional groups. The corresponding Coxeter graphs $\widetilde A_n (n\geqslant 1)$, $\widetilde B_n (n\geqslant 2)$, $\widetilde C_n (n\geqslant 2)$, $\widetilde D_n (n\geqslant 4)$, $\widetilde E_6$, $\widetilde E_7$, $\widetilde E_8$, $\widetilde F_4$ and $\widetilde G_2$ 
are displayed in Figure \ref{F:Coxeter2}. 

For $Z\in\{A,B,C,D,E,F,G\}$ and $n\in \mathbb N$, let $Z_n$ be the full subgraph of $\widetilde Z_n$ consisting of the black vertices --see Figure \ref{F:Coxeter2}. Then the Coxeter group $W_{Z_n}$ is finite. The graphs $Z_n$ and $\widetilde Z_n$ are known as \emph{Dynkin diagrams} and \emph{extended Dynkin diagrams} respectively (up to replacing the edges labeled~4 by a double edge and the edge labeled 6 by a triple edge). The meaning of the inequality signs will be explained later. From now on, we choose an arbitrary fixed extended Dynkin diagram~$\widetilde Z_n$ and we denote $W=W_{\widetilde Z_n}$ and $W_0=W_{Z_n}$. 

\end{definition}
\begin{figure}[h]
\begin{center}
\includegraphics{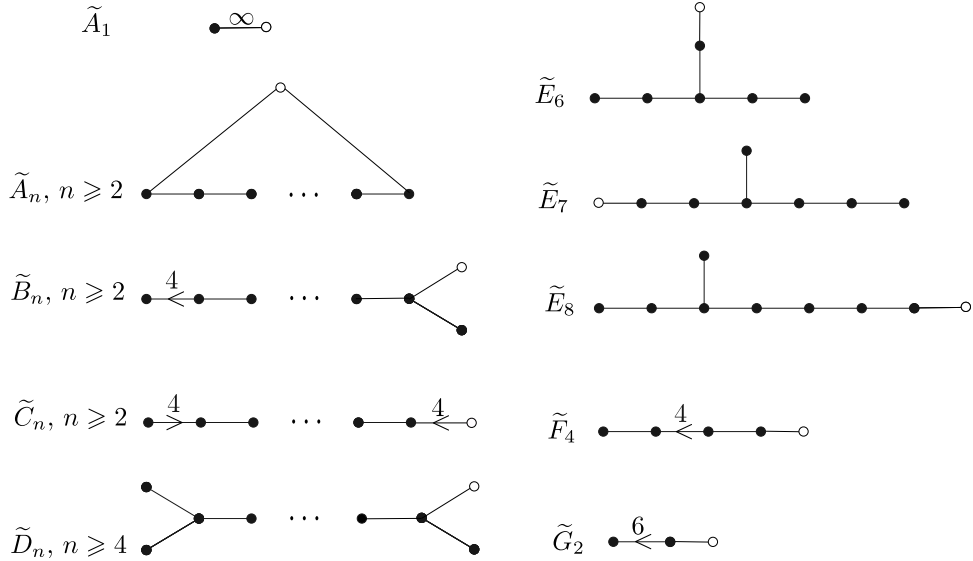}
\end{center}
\caption{Coxeter graphs  for irreducible euclidean Coxeter groups. For $Z\in \{A,B,C,D,E,F,G\}$ and $n\in \mathbb N$, the graph named $\widetilde Z_n$ has $n+1$ vertices. As usual, edge labels 3 are dropped.}
\label{F:Coxeter2}
\end{figure}

\begin{definition}[Root system]
The \emph{root system} of type $Z_n$ is described in \cite[Planches I to~IX]{Bourbaki}: this is a finite subset of $\mathbb E=\mathbb R^n$ and its elements are called \emph{roots}. Let us denote by $\Xi$ this root system; it contains a 
\emph{simple} system $\Xi'$ (a linear basis of $\mathbb E$ such that each $\alpha\in \Xi$ is a linear combination of $\Xi'$ with coefficients all of the same sign).
The set of (linear) reflections $r_{\alpha,0}$, $\alpha\in \Xi'$ is the set of standard generators of the finite Coxeter group $W_0$. 
 
\end{definition}

\begin{definition}[Highest root]
Given $\alpha\in \Xi$, the sum of the coefficients in the --unique-- expression of $\alpha$ as a linear combination of $\Xi'$ is called the \emph{height} of $\alpha$. There is a unique highest root; let us denote it by~$\mu$. We have a unique linear combination 
\begin{eqnarray}
\mu=\sum_{\alpha\in \Xi'} m_{\alpha}\alpha,
\label{Form:Mu0}
\end{eqnarray}
 where the $m_{\alpha}$ are positive integers. 
  \end{definition}

 \begin{definition}[Standard generators]
Let $r_{\mu,1}$ be the reflection in $\mathbb E$ through the hyperplane $\{\eta\in \mathbb E, \langle \eta, \mu\rangle=1\}$, let 
$${S=\{r_{\alpha,0}, \alpha\in \Xi'\}\cup\{r_{\mu,1}\}};$$ then $S$ is the set of standard generators for the euclidean Coxeter group $W$ and the reflection $r_{\mu,1}$ corresponds to the white vertex in the extended Dynkin diagram of Figure \ref{F:Coxeter2}. 
\end{definition}

\begin{definition}[Length of roots]
If $ \widetilde Z_n$ has no label on its edges, all roots in $\Xi$ have the same length while in the other cases, the roots have two different lengths and they are called \emph{long} or \emph{short}, accordingly. In presence of a label 4 (6, respectively), the ratio of the two different root lengths is $\sqrt{2}$ ($\sqrt{3}$, respectively). 
In the extended Dynkin diagram, the inequality sign(s) indicate(s) which roots are longer. 
\end{definition}

\begin{definition}[Elements of $W$, coroots and the Coxeter complex]
Every reflection in $W$ has the form $r_{\alpha,c}$ for $\alpha\in \Xi$ and $c\in \mathbb Z$; the corresponding  hyperplanes $H_{\alpha,c}$ provide a simplicial tiling of $\mathbb E$ called the \emph{Coxeter complex}. The spacing between two consecutive parallel hyperplanes $H_{\alpha,i}$ and $H_{\alpha,i+1}$ is given by the vector $\frac{\alpha}{\langle \alpha,\alpha\rangle}$, so that hyperplanes corresponding to long roots are more closely spaced. The set of translations in $W$ is generated by the translations of the form $t_{\alpha^{\vee}}=r_{\alpha,{i+1}}r_{\alpha,i}$ where $\alpha^{\vee}=\frac{2\alpha}{\langle\alpha,\alpha\rangle}$ for $\alpha\in \Xi$. The vector $\alpha^{\vee}$ is called the
 \emph{coroot} associated to~$\alpha$. 
\end{definition}

\subsection{Coxeter elements} \label{S:CoxeterElement}

\begin{definition}[Coxeter element]
A \emph{Coxeter element} for $W$ is a product of the elements in $S$ in any order. Every Coxeter element is a \emph{hyperbolic} isometry whose move-set is a non-linear affine hyperplane of $\mathbb E$ (\cite[Proposition 7.2]{McCFailure}). 
\end{definition}

From this point on, we make the additional assumption that $W$ is not $W_{\widetilde A_n}$ -- the case of $\widetilde A_n$ ($n\geqslant 1$) is somewhat different, as the extended Dynkin diagram is not a tree (except for $W_{\widetilde A_1}$ which is dihedral). The proof of the acylindrical hyperbolicity of the corresponding Artin-Tits group was already sketched in the introduction.

\begin{definition}[Bipartite Coxeter element]\label{D:Bipartite}
As the extended Dynkin diagram $\widetilde Z_n$ is a tree, there is a unique way of 2-coloring its vertices (say blue and green) in such a way that no two adjacent vertices have the same color. This yields a partition 
$S=S_{b}\sqcup S_{g}$ of $S$ in which the reflections in each part commute pairwise. Without loss, we may assume that $r_{\mu,1}\in S_{b}$.
Denote by $\Phi=\Xi'\sqcup \{\mu\}$; this set of roots is partitioned accordingly: $\Phi=(\{{\mu}\}\sqcup \Xi'_{b}) \sqcup \Xi'_{ g}$, where the vectors in each part are pairwise orthogonal. 
The two elements $\iota_{ b}=\Pi_{r\in S_b}r$ and $\iota_g=\Pi_{r\in S_g} r$ are involutions 
and we obtain two special Coxeter elements, inverses of each other --namely, $\iota_b\iota_g$ and $\iota_g\iota_b$-- called the \emph{bipartite Coxeter elements}.
From now on, we fix $w=\iota_b\iota_g$; this element will be referred to as \emph{the} Coxeter element. Also, we denote by $w_0$ the Coxeter element of $W_0$ defined by $w_0=r_{\mu,1}w$.

\end{definition}

\begin{definition}[Coxeter axis]\label{D:CoxeterAxis}
The min-set of the Coxeter element $w$ is a line called the \emph{Coxeter axis}; let us denote it by $L$. According to \cite[Remark 8.4]{McCFailure}, the direction of the Coxeter axis is given by the following:

\begin{eqnarray}
\gamma= \mu-\sum_{\alpha\in  \Xi'_b} m_{\alpha}\alpha=\sum_{\alpha\in  \Xi'_g} m_{\alpha}\alpha,
\label{Form:CoxeterAxis}
\end{eqnarray}
where the $m_{\alpha}$, $\alpha\in \Xi'$ are the positive integers involved in Formula (\ref{Form:Mu0}). For all $\eta\in L$, 
the displacement $\dis_w(\eta)$ under $w$ is the same and we denote it by $\gamma_0$ --see \cite[Proposition 3.7]{BMcC}; 
the standard form of $\Mov(w)$ is $\Mov(w)=\Dir(L)^{\perp}+\gamma_0$. 
\end{definition}

\begin{definition}[Horizontal and vertical]
A vector which is orthogonal to the direction of the Coxeter axis is called \emph{horizontal}; a vector which is not horizontal is called \emph{vertical}. Similarly, a reflection is called \emph{horizontal} (or \emph{vertical}, respectively) if its root is horizontal (vertical, respectively). 
\end{definition}

\begin{lemma}\label{L:SimpleNotHorizontal}
The standard generators of $W$ are vertical reflections. 
\end{lemma}

\begin{proof}
Let $\alpha\in\Phi$ be the root of an element in $S$. We need to check that $\alpha$ is not orthogonal to the direction of the Coxeter axis. 
Write $ \Phi=(\{\mu\}\sqcup \Xi'_b)\sqcup  \Xi'_g$ the bipartite decomposition of~$ \Phi$ as in Definition \ref{D:Bipartite}.
According to Formula (\ref{Form:CoxeterAxis}), we see that 
$$\langle \alpha,\gamma\rangle= \begin{cases} 
\langle \mu,\mu\rangle & \text{if $\alpha=\mu$,}\\ 
- m_{\alpha}\langle \alpha,\alpha\rangle & \text{if $\alpha\in \Xi'_b$,}\\
m_{\alpha}\langle \alpha,\alpha\rangle & \text{if $\alpha\in \Xi'_g$.}
\\
\end{cases}$$

because vectors in $\{\mu\}\sqcup \Xi'_b$ (in $\Xi'_g$, respectively) are pairwise orthogonal. The coefficients $m_{\alpha}$ in Formula~(\ref{Form:Mu0}) cannot be zero. 
It follows that $\alpha$ is not orthogonal to the direction of the Coxeter axis.
\end{proof}

\begin{definition}[Horizontal root system and horizontal factorizations] \label{D:Horizontal} See \cite[Definition 6.1]{McCS}.
Let $\Xi_h$ be the intersection of~$\Xi$ with the orthogonal complement of $\Dir(L)$. It turns out that 
$\Xi_h$ is a root system in $\Dir(L)^{\perp}$, called the \emph{horizontal root system}. The corresponding Coxeter group~$W_h$ is a subgroup  of $W_0$ called the \emph{horizontal Coxeter group}. 

The Coxeter element $w$
factorizes as $$w=r_{\mu,1}r_{\mu,0} w_h=t_{\mu^{\vee}}w_h,$$
where $w_h$ is a Coxeter element of $W_h$. 
%
%
Any factorization of $w$ of the form 
${w=t_{\lambda}r_1\ldots r_{n-1}}$ where~$t_{\lambda}$ is a translation and the $r_i$ are horizontal reflections is called a \emph{horizontal factorization}. In any horizontal factorization, the move-set of the horizontal part is precisely the horizontal hyperplane $\Dir(L)^{\perp}$ and we have $\Mov(w)=\Dir(L)^{\perp}+\lambda$. 
It follows that if $t_{\lambda}$ is the translation in a horizontal factorization, the projection of $\lambda$ on the vertical axis  $\Dir(L)$ is $\gamma_0$ (defined in Definition~\ref{D:CoxeterAxis}), independently of $\lambda$. 
\end{definition}

\begin{definition}[Translation in the direction of the Coxeter axis]\label{D:Translation}
The element $w_h$ has finite order --denote it by $e_0$-- and $w^{e_0}$ acts as a translation on $\mathbb E$ in the vertical direction (of vector $e_0\gamma_0$). 
For $u\in W$, we denote by $T_w(u)=w^{e_0}uw^{-e_0}$. 
If $r$ is a vertical reflection, then $T_w(r)$ is a vertical reflection through a distinct parallel hyperplane. If $r$ is a horizontal reflection, then $T_w(r)=r$. 
\end{definition}

\begin{proposition}[Translation part of $w$]\label{P:ShiftVertical} {\rm{See \cite[Proposition 6.3]{McCS}.}}
For every $i\in\mathbb Z$, we have $T_w(r_{\mu,i})=r_{\mu,i+1}$. 
\end{proposition}


%

\subsection{Reflection length and Garside structure}
\begin{definition}[Reflection length in $W$]
The reflection length $|u|_W$ of $u\in W$ is the minimal number of reflections in $W$ needed to express $u$. For $u,v\in W$, the relation $u\preccurlyeq_W v$ if and only if $|u|_W+|u^{-1}v|_W=|v|_W$ defines a partial order on $W$. The \emph{interval} $[1,v]^W$ is the set $\{u\in W, \ u\preccurlyeq_W v\}$. 
For any Coxeter element $w$, it turns out that $|w|_W=|w|_{\Isom}$,  $|u|_W=|u|_{\Isom}$ for every $u\in [1,w]^W$ and the restrictions of the orders $\preccurlyeq _W$ and $\preccurlyeq$ (on $\Isom(\mathbb E)$) to $[1,w]^W$ coincide. 
\end{definition}

\begin{proposition}{\rm{(Some elements in $[1,w]^W$){\bf{.}}}}
\begin{itemize}
\item[(i)] {\rm{\cite[Theorem 9.6]{McCFailure} and \cite[Definition 5.5]{McCS}.}} All vertical reflections lie in $[1,w]^W$ and $[1,w]^W$ contains exactly two horizontal reflections for each antipodal pair of horizontal roots in the root system $\Xi_h$.
\item[(ii)] {\rm{\cite[Proposition 6.3]{McCS}.}} The translations in $[1,w]^W$ are exactly those which appear in a horizontal factorization of $w$.
\end{itemize} \end{proposition}

\begin{definition}[Dual monoid and group]\label{D:Dual} The \emph{monoid associated} to $[1,w]^W$ is the monoid $M_w^W$ generated by $[1,w]^W$ subject to the relations $uu' = v$ whenever $u,u',v\in [1,w]^W$ satisfy $|u|_W+|u'|_W=|v|_W$ and $uu'=v$ in $W$. The \emph{dual Artin-Tits group} is the group with the same presentation; it is isomorphic to the Artin-Tits group $A$ associated to $W$ \cite[Theorem C]{McCS}. 
\end{definition}

\begin{theorem}[Lattice] \label{T:Lattice}
\begin{itemize}
\item[(i)]{\rm{\cite[Proposition 2.11]{McCS}.}} If the interval $[1,w]^W$ equipped with the restriction of the partial order~$\preccurlyeq_W$ is a lattice, then $M_w^W$ is a Garside monoid and the Artin-Tits group $A$ associated to $W$ is a Garside group. There is a monoid homomorphism (or a \emph{weight function}) $\rho: M_w^W\longrightarrow \mathbb N\cup\{0\}$ assigning 1 to each reflection; $w$ is the Garside element whose set of left and right divisors (the set of simple elements) is the interval $[1,w]^W$. 
\item[(ii)] {\rm{\cite[Theorem 10.3]{McCFailure}.}} The interval $[1,w]^W$ is a lattice if and only if the horizontal root system~$\Xi_h$ is irreducible. 
\item[(iii)] {\rm{\cite[Section 11]{McCFailure}.}} When $\Xi_h$ is not irreducible, it has $k_0$ irreducible components, with $k_0=2$ or $k_0=3$. The system $\Xi_h$ is irreducible if and only if $\widetilde Z_n\in \{\widetilde C_n, \widetilde G_2\}$. 
\end{itemize}
\end{theorem}

In order to deal with the cases where $[1,w]^W$ is not a lattice, McCammond and Sulway define a supergroup $C$ of $W$. They need first to introduce new isometries:

\begin{definition}[Factored translation] See \cite[Definition 6.7]{McCS}.
Suppose that the horizontal root system $\Xi_h$ is reducible; let $\Dir(L)^{\perp}=U_1\oplus \ldots\oplus U_{k_0}$ be the corresponding direct sum decomposition of the horizontal hyperplane --recall that $k_0\in\{2,3\}$ by Theorem \ref{T:Lattice}(iii). 
Let $t_{\lambda}$ be a translation in $[1,w]^W$. The projection of $\lambda$ onto the vertical line $\Dir(L)$ is $\gamma_0$ (see Definition~\ref{D:Horizontal}); for $i=1,\ldots,k_0$, let $\lambda_i$ be the projection of $\lambda$ onto the subspace $U_i$. The \emph{factored translations} corresponding to $\lambda$ are the $k_0$ translations $t_{\lambda_i+\frac{1}{k_0}\gamma_0}$. 
\end{definition}

\begin{definition}[Other groups]\label{D:Groups}
The \emph{crystallographic group} $C$ is the subgroup of $\Isom(\mathbb E)$ generated by $W$ together with the 
factored translations. The \emph{diagonal group} $D$ is the subgroup of $\Isom(\mathbb E)$ generated by the translations in $[1,w]^W$ together with horizontal reflections in $[1,w]^W$ and the \emph{factored group} $F$ is the subgroup of $\Isom(\mathbb E)$ generated by factored translations and horizontal reflections in $[1,w]^W$. A length is given so that a reflection has length 1, a factored translation has length $\frac{2}{k_0}$ and a translation has length 2. As for $W$, this yields a length and a partial order on the respective groups. 
The respective intervals $[1,w]^D$, $[1,w]^F$ and $[1,w]^C$ are naturally defined in the same way as $[1,w]^W$ (see for instance \cite[Section 4]{McCSurvey}) and one can also associate corresponding monoids and groups as in Definition \ref{D:Dual}. 
\end{definition}

\begin{theorem}[Garside] {\rm{See \cite[Proposition 7.4]{McCS} and \cite[Theorems A and B]{McCS}.}}
The interval $[1,w]^C$ is a balanced  lattice (which contains $[1,w]^W$). The monoid $M_w^C$ and the group~$\mathfrak C$ associated to $[1,w]^C$ are Garside and the Artin-Tits group $A$ associated to $W$ is a subgroup of $\mathfrak C$. The group~$\mathfrak C$ is called the \emph{crystallographic Garside group}. 
\end{theorem}

\begin{remark}\label{R:Weight}
$M_w^C$ is equipped with a monoid homomorphism $\rho$ extending the lengths given in Definition \ref{D:Groups};  
if $k_0=3$, $\rho$ takes values in $\mathbb N\cup \frac{2}{3}\mathbb{N}\cup\{0\}$: it sends each reflection to 1, and each factored translation to $\frac{2}{3}$ --see Remark \ref{R:Garside}(ii). In any case, reflections and factored translations are the atoms and $w$ is the Garside element. 
\end{remark}

\section{The loxodromic elements}\label{S:Loxodromic}

In this section, an irreducible euclidean Coxeter group $W$ distinct of $W_{\widetilde A_n}$ is fixed; we keep all notations from the previous section. 

The set $\mathcal S$ of simple elements of the crystallographic Garside group $\mathfrak C$ is in bijection with $[1,w]^C$ (and contains a copy of $[1,w]^W$). 
We shall use the same notation for a simple element in $\mathcal S$ and for the corresponding isometry of $\mathbb E$. For any simple element $s\in \mathcal S$, we denote by $\mathscr A(s)$ the set of atoms which left divide $s$. Since the interval $[1,w]^C$ is balanced, for each $s\in \mathcal S$, $\mathscr A(s)$ is also the set of atoms which right divide~$s$.

When dealing with elliptic isometries, the factored translations do not play an important role: 

\begin{lemma}\label{L:EllipticStartingSet}
Let $u\in[1,w]^W$ be an elliptic isometry. Then $$\mathscr A(u)= \{r, \text{$r$ is a reflection in $W$ whose fixed hyperplane contains $\Min(u)$}\}.$$
\end{lemma}
\begin{proof}
First, we shall see that no factored translation is in $\mathscr A(u)$. Suppose on the contrary that $\mathscr A(u)$ contains some factored translation $t_F$. 
Then by the first equality of \cite[Lemma 7.2]{McCS}, $u\in[1,w]^F$; by the second equality of \cite[Lemma 7.2]{McCS}, $u\in [1,w]^D$. Because all the factored translations have vectors with the  same vertical projection $\frac{\gamma_0}{k_0}$, 
$u$ has some vertical motion, so $u$ cannot be a product of only horizontal reflections. Therefore, there is a translation $t\in [1,w]^W$ so that $u=tu'$, with $|u'|_W=|u|_W-2$, which is impossible as $u$ was supposed to be elliptic --see Proposition \ref{P:Poset}(i). The fact that $\Min(u)$ is contained in the fixed hyperplane of every reflection in $\mathscr A(u)$ follows from Proposition~\ref{P:Poset}(ii). 
\end{proof}

In the sequel, we shall consider the set $\mathscr A(u)$ for different elliptic elements $u\in[1,w]^W$ and we will use Lemma \ref{L:EllipticStartingSet} without explicit reference. Also, according to our convention using the same symbol for an isometry in $[1,w]^C$ and the corresponding simple element in $\mathcal S$, we shall write, for $u\in [1,w]^C$, $\partial (u)=u^{-1}w$ and $\partial^{-1}(u)=wu^{-1}$. This is consistent with the notation in Definition~\ref{D:Complements}. In this context, the left-weightedness of a pair of simple elements $(s,s')$ is equivalent to $\mathscr A(\partial(s))\cap \mathscr A(s')=\emptyset$. Similarly, $(s,s')$ is right-weighted if and only if $\mathscr A(\partial^{-1}(s'))\cap \mathscr A(s)=\emptyset$.

Recall the ``translation'' $T_w$ defined by $T_w(u)=w^{e_0}u w^{-e_0}$ for all $u\in W$ (Definition \ref{D:Translation}). Note that 
for  $u\in [1,w]^W$, $\mathscr A(T_w(u))=T_w(\mathscr A(u))$.
Recall also the elements $\iota_b$ and $\iota_g$ --blue and green-- from Definition \ref{D:Bipartite}, which satisfy $\iota_b\iota_g=w$ (that is, $\partial(\iota_b)=\iota_g$ and $\partial^{-1}(\iota_g)=\iota_b$). 
In what follows we will denote $\iota'_b=T_w(\iota_b)$ and $\iota'_g=T_w(\iota_g)$, so that $w=\iota'_b\iota'_g$. Observe also that $\mathscr A(\iota_b)=\{r_{\mu,1}\}\sqcup \{r_{\alpha,0},\alpha\in \Xi'_b\}$ and $\mathscr A(\iota_g)=\{r_{\alpha,0}, \alpha\in \Xi'_g\}$. Finally, recall that $w_0$ is defined by $w_0=r_{\mu,1}w$.


\begin{lemma}\label{L:Loxod1}
The pair $(\iota'_b,w_0)$ is left and right-weighted. 
\end{lemma}
\begin{proof}
As $\mathscr A(\iota_b)=\{r_{\mu,1}\}\sqcup \{r_{\alpha,0},\alpha\in \Xi'_b\}$, we have
$$\mathscr A(\iota'_b)=\{T_w(r_{\mu,1})\}\sqcup \mathscr A'=\{r_{\mu,2}\}\sqcup \mathscr A',$$ 
(the equality $T_w(r_{\mu,1})=r_{\mu,2}$ comes from Proposition \ref{P:ShiftVertical}), where $\mathscr A'$ is a set of reflections through hyperplanes orthogonal to roots which are distinct from $\mu$.
Recall that $\partial^{-1}(w_0)=r_{\mu,1}$; the previous discussion shows that $r_{\mu,1}\notin \mathscr A(\iota'_b)$, whence $\mathscr A(\iota'_b)\cap \mathscr A(\partial^{-1}(w_0))=\emptyset$ and the right-weightedness follows. Also, $\partial(\iota'_b)=\iota'_g$. Fixed hyperplanes of reflections in $\mathscr A(\iota'_g)$ do not contain~0 while fixed hyperplanes of reflections in $\mathscr A(w_0)$ do contain~0. Therefore $\mathscr A(\iota'_g)\cap \mathscr A(w_0)=\emptyset$, whence left-weightedness.
\end{proof}

\begin{lemma}\label{L:Loxod2}
The pair $(w_0,\iota'_g)$ is left and right-weighted.
\end{lemma}

\begin{proof}
Let us describe $\partial(w_0)$. We claim that $\partial(w_0)$ is a reflection whose root is not in $\Phi$. 
 We have $w=r_{\mu,1}w_0=w_0(w_0^{-1}r_{\mu,1}w_0)$, whence $\partial(w_0)=w_0^{-1}r_{\mu,1}w_0$. But recall that 
$$w_0=\Pi_{\alpha\in\Xi'_b}r_{\alpha,0}\Pi_{\alpha\in\Xi'_g}r_{\alpha,0}.$$
In the first product, all reflections commute with $r_{\mu,1}$ and in the second (which is a product in which all reflections commute pairwise), all reflections commute with $r_{\mu,1}$, except one (as $r_{\mu,1}$ corresponds to a leaf in the extended Dynkin diagram).  Therefore, for some $\alpha\in \Xi'_g$, 
$w_0^{-1}r_{\mu,1}w_0=r_{\alpha,0}r_{\mu,1}r_{\alpha,0}$, which is a reflection whose root is not in $\Phi$. As all reflections in $\mathscr A(\iota'_g)$ have their roots in $\Xi'_g\subset \Phi$, we obtain $\mathscr A(\partial(w_0))\cap \mathscr A(\iota'_g)=\emptyset$, which shows left-weightedness. 

For right-weightedness, note that $\partial^{-1}(\iota'_g)=\iota'_b$. On the one hand, $\mathscr A(w_0)$ consists of reflections whose fixed hyperplane contains 0, on the other hand, $\mathscr A(\iota'_b)$ contains no reflection whose fixed hyperplane contains 0. Therefore $\mathscr A(\iota'_b)\cap \mathscr A(w_0)=\emptyset$ and we are done. 
\end{proof}

\begin{lemma}\label{L:Loxod3}
Let $r_v$ be a vertical reflection and let $\alpha\in \Xi$ be a vertical root. Then there is at most one $k\in \mathbb Z$ such that $r_{\alpha,k} r_v$ is a simple element. 
Also, there is at most one $l\in\mathbb Z$ such that $r_vr_{\alpha,l}$ is a simple element.
\end{lemma}
\begin{proof}
First, observe that for a pair of \emph{atoms} $r,r'$, $rr'$ is a simple element if and only if $r\in\mathscr A (\partial^{-1}(r'))$ if and only if $r'\in\mathscr A(\partial (r))$. 
By \cite[Lemma 9.3]{McCFailure}, $\partial^{-1}(r_v)$ is an elliptic isometry whose min-set is just a point.
There is at most one $k\in \mathbb Z$ such that $H_{\alpha,k}$ contains this point, that is, there is at most one
$k\in\mathbb Z$ such that $r_{\alpha,k}\in \mathscr A(\partial^{-1}(r_v))$. Similarly, $\partial (r_v)$ is an elliptic isometry whose min-set is just a point. 
There is at most one $l\in \mathbb Z$ such that $H_{\alpha,l}$ contains this point, that is, there is at most one $l\in \mathbb Z$ such that $r_{\alpha,l} \in \mathscr A(\partial(r_v))$. 
\end{proof}

\begin{lemma}\label{L:Loxod4}
There is a vertical reflection $r_0$ such that both  $(r_0,\iota'_b)$ and $(\iota'_g,r_0)$ are left and right-weighted. 
\end{lemma}
\begin{proof}
Fix a vertical root $\alpha$. Let $r_1, \ldots, r_p$ be an enumeration of $\mathscr A(\iota'_b)$ and let $s_1,\ldots, s_q$ be an enumeration of $\mathscr A(\iota'_g)$. By Lemma \ref{L:SimpleNotHorizontal} (and Definition \ref{D:Translation}), all these reflections are vertical. By Lemma \ref{L:Loxod3}, for each $i=1,\ldots,p$ and each $j=1,\ldots,q$, there is at most one 
$k_i$ such that $r_{\alpha, k_i}r_i$ is simple and at most one $l_j$ such that $s_jr_{\alpha,l_j}$ is a simple element. If we choose $m_0\notin\{k_1,\ldots,k_p,l_1,\ldots,l_q\}$, and $r_0=r_{\alpha,m_0}$, then $(r_{0}, \iota'_b)$ is left and right-weighted and $(\iota'_g,r_{0})$ is left and right-weighted.\end{proof}

\section{Proof of Theorem A}\label{S:Proofs}

In this section, an Artin-Tits group $A$ of euclidean type distinct from the affine braid group is fixed. We keep notations from the previous sections with the following exception. As it is a standard notation for the Garside element in Garside groups, we will use the letter $\Delta$ for the Garside element of $\mathfrak C$ --this is the same that was denoted above by $w$.

\begin{definition}\label{D:Loxod}
For the remaining of the paper, we define the following element of $\mathfrak C$, which is also an element of $A$. Let $r_0$ be as in Lemma \ref{L:Loxod4}.  
Define 
$$x= r_0\cdot \iota'_b\cdot w_0 \cdot \iota'_g \cdot r_0.$$
\end{definition}

First, we gather some facts about the element $x$. 
For any $g\in \mathfrak C$ with $\inf(g)=0$, we denote $\partial(g)=g^{-1}\Delta^{\sup(g)}$ --this matches the notation for the right complement of a simple element $s$ (in which case $\sup(s)=1$). 

\begin{proposition}\label{P:LoxodromicX}

\begin{itemize}
\item[(i)] The left and right normal form of $x$ are the same and we just call it the ``normal form''; this normal form is given by the formula in Definition \ref{D:Loxod}. 
\item[(ii)] The first and last factor of the normal form of $x$ coincide; thus $x$ is \emph{rigid}: for every $m\in \mathbb N$, the left --and right-- normal form of $x^m$ is the concatenation of $m$ copies of the normal form of $x$. 
\item[(iii)] Both normal forms of $x$ and $\partial(x)$ contain a factor which is the right complement of a reflection; the elements $x$ and $\partial(x)$ are not absorbable. 
\item[(iv)] For each $m\geqslant 0$, $\partial(x^m)=\Pi_{i=0}^{m-1} \tau^{5i}(\partial(x))$ and this is in left and right normal form as written. 
\item[(v)] No non-trivial power of $\Delta$ commutes with $x$. 
\end{itemize}
\end{proposition}

\begin{proof}
(i) follows from Lemmas \ref{L:Loxod1}, \ref{L:Loxod2} and \ref{L:Loxod4}; (ii) is immediate. (iii) To see that $x$ is not absorbable, it suffices to notice  that $w_0$ is not absorbable and to use Lemma \ref{L:AbsorbableSubword}. Recall that $\rho$ is the weight function of the monoid $M_w^C$ (Remark \ref{R:Weight}). 
If $w_0$ was absorbable, we would have some $s\in \mathcal S$ such that $w_0s$ is a \emph{proper} simple element. We then would have $\rho(w_0s)=\rho(w_0)+\rho(s)=n+\rho(s)<n+1$. Then $\rho(s)<1$ and the only possibility is that $k_0=3$ and $\rho(s)=\frac{2}{3}$. But then $\rho(\partial(w_0s))=\frac{1}{3}$, which is impossible since there is no simple element with weight $\frac{1}{3}$. For the same reason, $\partial (r_0)$ is not absorbable and $\partial (x)$ is not absorbable. 
(iv) In any Garside group, if $x_1\ldots x_q$ is a left and right normal form, then $\partial(x_q)\ldots \tau^{q-1}(\partial(x_1))$ is also a left and right normal form. (v) Otherwise, there would be some power $l\neq 0$ of $\Delta$ commuting with $r_0$ (see \cite[Proposition 2.14]{McCS}) and hence $r_0$ would also commute with $\Delta^{le_0}$ ($e_0$ is given in Definition \ref{D:Translation}), which is impossible as $r_0$ is vertical and all isometries $T_w^k(r_0), k\in \mathbb Z$ are distinct (Definition \ref{D:Translation}).  
\end{proof}

\subsection{Loxodromic}

\begin{theorem}\label{T:Loxodromic}
The element $x$ acts in a loxodromic way on the additional length graph $\mathcal C_{AL}(\mathfrak C)$. More precisely, $d_{AL}(\ast, X^k)\geqslant \frac{|k|}{2}$ for all $k$ in $\mathbb Z$. As a consequence, $\mathcal C_{AL}(\mathfrak C)$ has infinite diameter. 
\end{theorem}

Here, $X^k$, $k\in \mathbb Z$ stands for the vertex $x^k\Delta^{\mathbb Z}=x^k\cdot \ast$ of $\mathcal C_{AL}(\mathfrak C)$. 
Throughout, we shall use the symbol $\preccurlyeq$ for the order on $\mathfrak C$ which extends the order on $M_w^C$, itself extending the order on $[1,w]^C$: for $g,h\in \mathfrak C$, $g\preccurlyeq h$ if and only if $g^{-1}h\in M_w^C$. 
The main tool for the proof is a projection map from $\mathcal C_{AL}(\mathfrak C)$ to $\{X^k,k\in \mathbb Z\}$:  

\begin{lemma}  {\rm{See \cite[Definition 3]{CalvezWiestAH}.}} \label{L:Projection}
There is a well-defined map $\Lambda$ from the set of vertices of $\mathcal C_{AL}(\mathfrak C)$ to $\mathbb Z$ given by the formula 
$$\Lambda(V)=-\max\{k\in \mathbb Z,\ x\not\preccurlyeq \underline{x^k\cdot V}\}.$$
This yields a projection 
from the set of vertices of $\mathcal C_{AL}(\mathfrak C)$ onto the set of vertices $\{X^k, k\in \mathbb Z\}$, given by $V\mapsto X^{\Lambda(V)}$.
\end{lemma}

This projection has the following key-property:
\begin{proposition}\label{P:Contracting} {\rm{See \cite[Proposition 4]{CalvezWiestAH}.}} \label{P:Contract}
Let $V_1,V_2$ be two vertices of $\mathcal C_{AL}(\mathfrak C)$; let $\Lambda_1=\Lambda(V_1)$ and $\Lambda_2=\Lambda(V_2)$. Suppose that $\Lambda_2-\Lambda_1\geqslant 3$. Then the preferred path $\mathcal A(V_1,V_2)$ contains the subpath $\mathcal A(X^{\Lambda_1+1},X^{\Lambda_2-1})$. 
\end{proposition}

The technical ground for proving 
Proposition \ref{P:Contract} is achieved in \cite[Lemmas 5,6,7]{CalvezWiestAH}. The proof of these lemmas is unchanged in our context, with the exception that 
$\Delta^{\sup(x)}$ is not central, so $\partial(x^k)\neq \partial(x)^k$ for $k\in\mathbb N$.  As a consequence, the expression ``$k$ copies of the normal form of $\partial x$'' in \cite{CalvezWiestAH} must be replaced by ``the normal form of $\partial (x^k)$''. 

{\it{Proof of Theorem \ref{T:Loxodromic}.}} See the proof of \cite[Proposition 5]{CalvezWiestAH}.\hfill $\Box$

\subsection{WPD}
\begin{theorem}\label{T:WPD}
The action of $x\in A<\mathfrak C$ on $\mathcal C_{AL}(\mathfrak C)$ is WPD, that is, for each vertex $V$ of $\mathcal C_{AL}(\mathfrak C)$ and for each $\kappa>0$, there exists an integer $N$ such that the set
 $$\{g\in \mathfrak C, \ d_{AL}(V,g\cdot V)\leqslant \kappa, \ d_{AL}(x^N\cdot V, gx^N\cdot V)\leqslant \kappa\}$$ is finite.

Equivalently, for each $\kappa>0$, there exists an integer $N$ such that the set 
$$\{g\in \mathfrak C, \ d_{AL}(\ast,g\cdot \ast)\leqslant \kappa, \ d_{AL}(X^N, g\cdot X^N)\leqslant \kappa\}$$ is finite.  
\end{theorem}

The proof follows very closely the proof of \cite[Proposition 6]{CalvezWiestAH}. 
The first step is to see that the projection $\Lambda$ defined in Lemma \ref{L:Projection} is coarsely Lipschitz. In what follows, $K$ is the maximal Hausdorff distance between a geodesic and a preferred path between a pair of vertices of $\mathcal C_{AL}(\mathfrak C)$; one can take $K=39$ --see Theorem~\ref{T:CAL}. 
\begin{lemma}\label{L:Lipschitz}{\rm{\cite[Proposition 7]{CalvezWiestAH}.}}
Suppose that $V_1,V_2$ are vertices of $\mathcal C_{AL}(\mathfrak C)$. 
Then 
$$|\Lambda(V_2)-\Lambda(V_1)|\leqslant 2(d_{AL}(V_1, V_2) + 2K + 1).$$ 
\end{lemma}

{\it{Proof of Theorem \ref{T:WPD}.}}
Fix any $\kappa>0$; let $\xi=\kappa+2K+1$ and fix $N\geqslant 4\xi+3$. Suppose that $g\in \mathfrak C$ satisfies

\begin{eqnarray}\label{E:WPD0}
\begin{cases} d_{AL}(\ast,g\cdot \ast)\leqslant \kappa,\\
d_{AL}(X^N, g\cdot X^N)\leqslant \kappa.
\end{cases}
\end{eqnarray}

\begin{claim}\label{CLAIM}
We have 
\begin{eqnarray}\label{Distance}
\Lambda(g\cdot X^N)-\Lambda(g\cdot \ast)\geqslant N-4\xi\geqslant 3. 
\end{eqnarray}
The preferred path between $g\cdot \ast$ and $g\cdot X^N$ contains the subpath $\mathcal A(X^{\Lambda(g\cdot \ast)+1},X^{\Lambda(g\cdot X^N)-1})$.
\end{claim}

\begin{proof}
Under our hypothesis, we have according to Lemma \ref{L:Lipschitz} (observe that $\Lambda(X^N)=N$ and $\Lambda(\ast)=0$), 
$$\begin{cases} 
|N-\Lambda(g\cdot X^N)| & \leqslant 2(\kappa+2K+1)=2\xi,\\
|\Lambda(g\cdot \ast)| & \leqslant 2(\kappa+2K+1)=2\xi;
\end{cases}$$
the first part of the claim follows. 
Due to our choice of $N$, the second part of the claim follows immediately from Proposition \ref{P:Contracting}.
\end{proof}

Write for short $\Lambda_1=\Lambda(g\cdot\ast )$ and $\Lambda_2=\Lambda(g\cdot X^N)$. By Proposition \ref{P:Preferred}(iii), the preferred path $\mathcal A(g\cdot \ast,g\cdot X^N)$ is the $g$ left translate of the preferred path $\mathcal A( \ast,X^N)$, so by Claim \ref{CLAIM}, there are vertices $P$ and~$Q$ along $\mathcal A(\ast,X^N)$ so that $g\cdot P=X^{\Lambda_1+1}$ and $g\cdot Q=X^{\Lambda_2-1}$. 

\begin{claim}\label{Claim}
$P$ is represented by some power of $x$, say $P=X^a$ (and thus $Q=X^{a+\Lambda_2-\Lambda_1-2}$). 
\end{claim}

\begin{proof}
Write $x=x_1\ldots x_5$ for the normal form of $x$ (see Definition \ref{D:Loxod}) and recall that $x$ is rigid (Proposition \ref{P:LoxodromicX}(ii)). Notice that the distinguished representatives of the $(5N+1)$ vertices along the path $\mathcal A(\ast,X^N)$ are 
$$\ast, x_1,x_1x_2,x_1x_2x_3,x_1x_2x_3x_4,x, xx_1,\ldots, x^2,\ldots X^N.$$
We know that the path $\mathcal A(P,Q)$ has the same length as the path $\mathcal A(X^{\Lambda_1+1},X^{\Lambda_2-1})$, that is $5(\Lambda_2-\Lambda_1-2)$, so if $P=X^a$ then $Q=X^{a+\Lambda_2-\Lambda_1-2}$.

  Suppose, in contradiction with the claim, that $P=(x^a x_1\ldots x_{j})\Delta^{\mathbb Z}$ for some $a$ and $1\leqslant j \leqslant 4$. Then as the length of $\mathcal A(P,Q)$ is $5(\Lambda_2-\Lambda_1-2)$, $Q$ must be $(x^{a+\Lambda_2-\Lambda_1-2}x_1\ldots x_j)\Delta^{\mathbb Z}$.

  Therefore we have (notice that $\underline P$ is a prefix of $\underline Q$)
\begin{eqnarray}\label{E:WPD}
\underline P^{-1}\underline Q= x_{j+1}\ldots x_5x^{\Lambda_2-\Lambda_1-3}x_1\ldots x_{j}, 
\end{eqnarray}
and this is the left normal form, by construction of $x$. 

On the other hand, 
$$x^{\Lambda_1+1}=\underline{g\cdot P}=g\underline P \Delta^{-\inf(g\underline P)}$$ and
$$ x^{\Lambda_2-1}=\underline{g\cdot Q}=g\underline Q\Delta^{-\inf(g\underline Q)},$$
from which we deduce (recall that $\tau$ is conjugation by $\Delta$)
$$x^{\Lambda_2-\Lambda_1-2}=\Delta^{\inf(g\underline P)}\underline P^{-1}g^{-1}g\underline Q \Delta^{-\inf(g\underline Q)}=\tau^{-\inf(g\underline P)}({\underline P}^{-1}\underline Q)\Delta^{\inf(g\underline P)-\inf(g\underline Q)}.$$
Considerations on the infimum show that $\inf(g\underline P)=\inf(g\underline Q)$ and we see also that $x^{\Lambda_2-\Lambda_1-2}$ and  ${\underline P}^{-1}\underline Q$ are conjugate by $\Delta^{-\inf(g\underline P)}$. By \cite[Proposition 2.14]{McCS} (conjugation of normal forms by~$\Delta$), we see in particular that the first factor (last factor, respectively) of the normal form of $x^{\Lambda_2-\Lambda_1-2}$ and the first factor (last factor, respectively) of the normal form of ${\underline{P}}^{-1}\underline Q$ are conjugate by $\Delta^{-\inf(g\underline P)}$.

By construction of $x$, the first and last factor of the normal form of $x^{\Lambda_2-\Lambda_1-2}$ are $x_1$ and $x_5$ respectively. In view of Equation (\ref{E:WPD}) the first and last factor of the normal form of ${\underline{P}}^{-1}\underline Q$ are $x_{j+1}$ and $x_j$ respectively. As $\tau$ commutes with the weight function $\rho$, we obtain $\rho(x_{j+1})=\rho(x_1)=1$ and $\rho(x_{j})=\rho(x_5)=1$, by construction of $x$, which contradicts the choice of $j$. 
\end{proof}


\begin{claim}\label{C:Power}
Any element $g\in \mathfrak C$ which satisfies the conditions (\ref{E:WPD0}) is a power of $x$. 
\end{claim}
\begin{proof}
As $\Lambda_2-\Lambda_1\geqslant 3$, we have along the path $\mathcal A(P,Q)=\mathcal A(X^a, X^{a+\Lambda_2-\Lambda_1-2})$ {\it{at least two}} consecutive vertices $X^a$ and $X^{a+1}$ such that 
$$\begin{cases}
g\cdot X^a=X^{b},\\
g\cdot X^{a+1}=X^{b+1},
\end{cases}$$ 
(with $b=\Lambda_1+1$). 
 By the first equality, we obtain $gx^a=x^b\Delta^{l}$ for some $l\in \mathbb Z$. By the second equality, we obtain $gx^{a+1}=x^{b+1}\Delta^{l'}$ for some $l'\in \mathbb Z$. Combining both assertions yields 
 $x^b\Delta^l x= x^{b+1}\Delta^{l'}$, which is equivalent to $\Delta^l x=x\Delta^{l'}$. 
This forces $l=l'$ (by considering the infimum) and also $l=0$ since no non-trivial power of $\Delta$ commutes with~$x$ (Proposition \ref{P:LoxodromicX}(v)). We deduce that $g=x^{b-a}$ is a power of $x$ as desired. 
\end{proof}

By Lemma \ref{L:Lipschitz}, if $d_{AL}(\ast, x^{l}\cdot \ast)\leqslant \kappa$, we have $|l|\leqslant 2\xi$; hence Claim \ref{C:Power} shows that the elements of $\mathfrak C$ which satisfy the conditions (\ref{E:WPD0}) are contained in the \emph{finite} set $\{x^{l}, |l|\leqslant 2\xi\}$. 
This achieves the proof of Theorem \ref{T:WPD}. \hfill $\Box$\\

{\it{Proof of Theorem A.}} This follows from Theorem \ref{T:WPD} and Osin's \cite[Theorem 1.2]{Osin}. \hfill $\Box$\\

With the same line of arguments, we have also shown:

\begin{CorollaryB}
The crystallographic Garside group $\mathfrak C$ is acylindrically hyperbolic. 
\end{CorollaryB}

{\bf{Acknowledgements.}} A great part of this work was done during COVID lockdown; I am deeply indebted to extra-official support of PYC. The initial impetus for learning about euclidean Artin-Tits groups has its roots in conversations with Bruno Cisneros to whom I am very grateful. The present paper builds on joint work with Bert Wiest whom I wish to thank for useful discussions. Heartfelt thanks also go to Mar\'ia Cumplido for a careful reading of an earlier draft and many useful comments. 
The author was supported by FONDEYT Regular 1180335, MTM2016-76453-C2-1-P and FEDER. 


\begin{thebibliography}{99}
\bibitem{Bestvina} M. Bestvina, {\it{Non-positively curved aspects of Artin groups of finite type}}, Geom. Top. 3 (1999), 269--302.

\bibitem{BestvinaFujiwara} M. Bestvina, K. Fujiwara, {\it{Bounded cohomology of subgroups of mapping class groups}}, Geom. Top. 6 (2002), 69--89. 

\bibitem{Bourbaki} N. Bourbaki, {\it \'El\'ements de math\'ematiques}. Fasc. XXXIV. Groupes et alg\`ebres de Lie. Chapitre~IV: Groupes de Coxeter et syst\`emes de Tits. 
Chapitre V: Groupes engendr\'es par des r\'efl\'exions. Chapitre VI: Syst\`emes de racines.
 Actualit\'es Scientifiques et Industrielles, No. 1137, Hermann, Paris, 1968.


\bibitem{Bowditch} B. Bowditch, {\it{Tight geodesics in the curve complex}}, Invent. Math. 171 (2008), 281--300.


\bibitem{BMcC} N. Brady, J. McCammond, {\textit{Factoring Euclidean isometries}}, Int. J. Alg. Comp. 25, 2015, 325--347.

\bibitem{BrieskornSaito} E. Brieskorn, K. Saito, \textit{Artin-Gruppen und Coxeter-Gruppen}, Invent. Math. 17 (1972), 245--271.

\bibitem{CalvezWiestCurve} M. Calvez, B. Wiest, {\it{Curve graphs and Garside groups}}, Geom. Ded. 188 (1) (2017), 195--213.

\bibitem{CalvezWiestAH} M. Calvez, B. Wiest, {\textit{Acylindrical hyperbolicity and Artin-Tits groups of spherical type}}, Geom. Ded. 191 (1), 2017, 199--215.

\bibitem{CharneyMorris} R. Charney, R. Morris-Wright, {\it{Artin groups of infinite type: Trivial centers and acylindrical hyperbolicity}}, Proc. Amer. Math. Soc. 147 (2019), 3675-3689.

\bibitem{DDGKM} P. Dehornoy, F. Digne, E. Godelle, D. Krammer, J. Michel, {\textit{Foundations of Garside Theory}}, EMS Tracts in Mathematics, volume 22, European Mathematical Society, 2015.

\bibitem{Deligne} P. Deligne, \textit{Les immeubles des groupes de tresses g\'{e}n\'{e}ralis\'{e}s}, Invent. Math. 17 (1972), 273--302.

\bibitem{DigneA} F. Digne, {\textit{Pr\'esentations duales des groupes de tresses de type affine $\widetilde A$}}, Comment. Math. Helvet. 81 (1), 2006, 23--47.

\bibitem{DigneC} F. Digne, {\textit{A Garside presentation for Artin-Tits groups of type $\widetilde C_n$}}, Ann. Inst. Fourier 62 (2), 2012, 641--666. 


\bibitem{Haettel} T. Haettel, {\it{XXL type Artin groups are CAT(0) and acylindrically hyperbolic}}, arXiv:1905.11032.


\bibitem{Humphreys} J. E. Humphreys, {\textit{Reflection groups and Coxeter groups}}, Cambridge Studies in Advanced Mathematics 29, Cambridge University Press, 1990.


\bibitem{KentPeifer} R. P. Kent, D. Peifer, {\textit{A geometric and algebraic description of annular braid groups}}, Int. J. Alg. Comp. 12, 2002, 85--97.

\bibitem{KimKoberda} S.-H. Kim. T. Koberda, {\it{The geometry of the curve graph of a right-angled Artin group}}, Int. J. Alg. Comp. 24 (2) (2014), 121--169. 

\bibitem{MasurMinsky} H. Masur, Y. Minsky, {\it{Geometry of the complex of curves. I. Hyperbolicity}}, Invent. Math. 138 (1999), 103--149. 

\bibitem{McCGarside} J. McCammond, {\textit{An introduction to Garside structures}}, available at http://web.math.ucsb.edu/~jon.mccammond

\bibitem{McCFailure} J. McCammond, {\textit{Dual euclidean Artin groups and the failure of the lattice property}}, J. Algebra 437, 2015, 308--343.

\bibitem{McCSurvey} J. McCammond, {\textit{The structure of euclidean Artin groups}},  Geometric and cohomological group theory (London Mathematical Society Lecture Note Series), Kropholler, P., Leary, I., Martinez, C.,  Nucinkis, B. (Eds.). Cambridge University Press, 2017, 82--114.

\bibitem{McCS} J. McCammond, R. Sulway, {\textit{Artin groups of euclidean type}}, Invent. Math. 210, 2017, 231--282.

\bibitem{Osin} D. Osin, {\it{Acylindrically hyperbolic groups}}, Trans. Amer. Math. Soc. 368, 2016, 851--888.

\bibitem{MartinPrz} A. Martin, P. Przytycki, {\it{Acylindrical actions for two-dimensional Artin groups}}, arXiv:1906.03154.

\bibitem{Squier} C. Squier, {\it{On certain 3-generator Artin groups}}, Trans. Amer. Math. Soc. 302 (1), 1987, 117--124.

\bibitem{Tits} J. Tits, {\it{Normalisateurs de tores. I. Groupes de Coxeter \'etendus}}, J. Algebra 4 (1966), 96--116.

\bibitem{Vaskou} N. Vaskou, {\it{Acylindrical hyperbolicity for Artin groups of dimension 2}}, arXiv:2007.16169.

\end{thebibliography}
\end{document}